\pdfoutput=1

\documentclass{amsart}

\usepackage{amsfonts,amssymb,amsmath,enumerate,verbatim,mathtools,tikz,bm,mathrsfs,tikz-cd,hyperref,comment,stmaryrd,amsthm}
\usepackage{colonequals}
\usepackage{adjustbox}
\usepackage[all,pdf]{xy}
\hypersetup{
    colorlinks=true,
    linkcolor=blue,
    filecolor=magenta,      
    urlcolor=cyan,
}

\makeatletter
\@namedef{subjclassname@2020}{\textup{2020} Mathematics Subject Classification}
\makeatother

\author[Jian Liu]{Jian Liu}
\address{School of Mathematical Sciences, Shanghai Jiao Tong University, Shanghai 200240, P.R. China. }
\email{liuj231@sjtu.edu.cn}


\DeclareMathOperator{\cone}{cone}

\DeclareMathOperator{\depth}{depth}

\DeclareMathOperator{\h}{H}

\newcommand{\CC}{\mathcal{C}}

\newcommand{\Z}{\mathbb{Z}}

\newcommand{\T}{\mathcal{T}}

\newcommand{\D}{\mathsf{D}}

\newcommand{\N}{\mathbb{N}}

\newcommand{\para}{\mathbin{\!/\mkern-5mu/\!}}

\pgfdeclarelayer{bg}    
\pgfsetlayers{bg,main}

\newcommand{\del}{\partial}
\newcommand{\f}{{\bm{f}}}

\newcommand{\y}{{\bm{y}}}

\newcommand{\0}{{\bf{0}}}
\newcommand{\m}{\mathfrak{m}}

\newcommand{\p}{\mathfrak{p}}
\newcommand{\q}{\mathfrak{q}}

\DeclareMathOperator{\pd}{pd}

\DeclareMathOperator{\ca}{ca}

\DeclareMathOperator{\jac}{jac}

\DeclareMathOperator{\id}{id}
\DeclareMathOperator{\Spec}{Spec}

\DeclareMathOperator{\ann}{ann}

\DeclareMathOperator{\Hom}{Hom}

\DeclareMathOperator{\Ext}{Ext}

\DeclareMathOperator{\End}{End}

\DeclareMathOperator{\V}{V}

\DeclareMathOperator{\thick}{\mathsf{thick}}

\DeclareMathOperator{\Sing}{Sing}
\DeclareMathOperator{\Supp}{Supp}

\newcommand{\per}{\mathsf{perf}}
\newcommand{\sg}{\mathsf{sg}}

\newcommand{\mo}{\mathsf{mod}}

\newtheorem{theorem}{Theorem}[section]

\newtheorem{proposition}[theorem]{Proposition}

\newtheorem{lemma}[theorem]{Lemma}
\newtheorem{corollary}[theorem]{Corollary}

\theoremstyle{definition}
\newtheorem{example}[theorem]{Example}
\newtheorem{remark}[theorem]{Remark}

\newtheorem{chunk}[theorem]{}

\usepackage[utf8]{inputenc}

\newtheorem*{ack}{Acknowledgements}

\title[Annihilators and dimensions of the singularity category]{Annihilators and dimensions of the singularity category}
\date{\today}

\begin{document}
\maketitle

\begin{abstract}
    Let $R$ be a commutative noetherian ring. We prove that if $R$ is either an equidimensional finitely generated algebra over a perfect field, or an equidimensional equicharacteristic complete local ring with a perfect residue field, then the annihilator of the singularity category of $R$ coincides with the Jacobian ideal of $R$ up to radical. We establish a relation between the annihilator of the singularity category of $R$ and  the cohomological annihilator of $R$ under some mild assumptions. Finally, we give an upper bound for the dimension of the singularity category of an equicharacteristic excellent local ring with isolated singularity. This extends a result of Dao and Takahashi to non Cohen-Macaulay rings. 
\end{abstract}
\section{Introduction}
Let $R$ be a commutative noehterian ring. The \emph{singularity category} of $R$, denoted $\D_{\sg}(R)$, is the Verdier quotient of the bounded derived category with respect to the full subcategory of perfect complexes. This was introduced by Buchweitz \cite{Buchweitz1987} under the name ``stable derived category'' and later also by Orlov \cite{Orlov04,Orlov09} who related the singularity category to the homological mirror symmetry conjecture. The terminology is justified by the fact: $\D_{\sg}(R)$ is trivial if and only if $R$ is regular.
For a strongly Gorenstein ring $R$ (i.e. $R$ has finite injective dimension as an $R$-module), Buchweitz \cite{Buchweitz1987} established a triangle equivalence between the singularity category of $R$ and the stable category of maximal Cohen-Macaulay $R$-modules. 

In this article, we focus on studying the \emph{annihilator of the singularity category} of $R$, namely an ideal of $R$ consisting of elements in $R$ that annihilate the endomorphism ring of all complexes in $\D_{\sg}(R)$; see \ref{def}. We denote this ideal by $\ann_R \D_{\sg}(R)$. This ideal  measures the singularity of $R$ in the sense that $R$ is regular if and only if $\ann_R\D_{\sg}(R)=R$; see Example \ref{regular}. 

Buchweitz \cite{Buchweitz1987} observed that the Jacobian ideal $\jac(R)$ of $R$ annihilates the singularity category of $R$ when $R$ is a quotient of a formal power series over a field modulo a regular sequence. 
Recently, this result was extended to a large family of rings (e.g. equicharacteristic complete Cohen-Macaulay local ring) by Iyengar and Takahashi \cite{IT2021}. There is also a result contained in \cite{IT2021}: a power of the generalized  Jacobian ideal annihilates the  singularity category of a commutative noetherian ring; we point out this result should have an equidimensional assumption (see Example \ref{fail}).

It is worth noting that there are only a few classes of rings whose annihilators of the singularity category are known.  When $R$ is a one dimensional reduced complete Gorenstein local ring, Esentepe \cite{Esentepe} proved that the annihilator $\ann_R\D_{\sg}(R)$ is the conductor ideal of $R$, namely the annihilator of $\overline{R}/R$ over $R$, where $\overline{R}$ is the integral closure of $R$ inside its total quotient ring.

Our first result concerns the connection between the Jacobian ideal $\jac(R)$ and the ideal $\ann_R \D_{\sg}(R)$.
\begin{theorem}\label{t1} (see \ref{main result}) 
Let $R$ be either an equidimensional finitely generated algebra over a perfect field, or an equidimensional equicharacteristic complete local ring with a perfect residue field. Then
 $$
 \sqrt{\jac(R)}=\sqrt{\ann_R\D_{\sg}(R)}.
 $$
 In particular, $\jac(R)^s$ annihilates the singularity category of $R$ for some integer $s$.
\end{theorem}
The proof of the above result relies on the Jacobian criterion and Theorem \ref{locus}. It is proved in Theorem \ref{locus} that $\ann_R\D_{\sg}(R)$ defines the singular locus of $R$ if $\D_{\sg}(R)$ has a \emph{strong generator}; see definition of strong generator in \ref{def of dim}. The proof of Theorem \ref{locus} makes use of the localization and annihilator of an essentially small $R$-linear triangulated category discussed in Section \ref{section2}. The hypothesis of Theorem \ref{t1} ensures that $\D_{\sg}(R)$ has a strong generator.
Indeed, this can be inferred from a result of Iyengar and Takahashi \cite{IT2016} that says the bounded derived category of $R$ has a strong generator if $R$ is either a localization of a finitely generated algebra over a field or an equicharacteristic excellent local ring.

The ideal $\ann_R \D_{\sg}(R)$ is closely related the \emph{cohomological annihilator} $\ca(R)$ of $R$. By definition, $\ca(R)=\bigcup_{n\in \Z}\ca^n(R)$, where $\ca^n(R)$ consists of elements $r$ in $R$ such that $r\cdot \Ext^n_R(M,N)=0$ for all finitely generated $R$-modules $M,N$. The ideal $\ca(R)$ was initially studied by
Dieterich \cite{Dieterich} and Yoshino \cite{Yoshino87} in connection with the Brauer-Thrall conjecture. 
Cohomological annihilators are of independent interest and have been systematically studied by Wang \cite{Wang1994, Wang1998}, Iyengar and Takahashi \cite{IT2016, IT2021}.
When $R$ is a strongly Gorenstein ring,  Esentepe \cite{Esentepe} observed that the cohomological annihilator coincides with the annihilator of the singularity category. We compare the relation of these two annihilators in Section \ref{section4} for general rings. The main result in Section \ref{section4} is the following:

\begin{proposition}\label{t2} (see \ref{relation})
Let $R$ be a commutative noetherian ring. Then

(1) $\ca(R)\subseteq \ann_R \D_{\sg}(R)$.

 (2) If furthermore $R$ is either a localization of a finitely generated algebra over a field or an equicharacteristic excellent local ring, then 
 $$
\sqrt{\ca(R)}=\sqrt{ \ann_R\D_{\sg}(R)}.
$$
\end{proposition}

For a local ring $R$, it is proved that the cohomological annihilator contains the socle of $R$; see \cite{IT2016}. Hence in this case, Proposition \ref{t2} yields that the socle of $R$ annihilates the singularity category of $R$; see Corollary \ref{socle}.

Let $G$ be an object in a triangulated category $\T$, the \emph{generation time} of $G$ in $\T$ is the minimal number of cones required to generate $\T$, up to shifts and direct summands; see \ref{def of dim}. If there exists an object $G$ in $\T$ with finite generation time, then this number will give an upper bound for the dimension of $\T$ introduced by Rouquier \cite{Rouquier}. By making use of the  dimension of the stable category of exterior algebras, Rouquier \cite{Rouquier06} proved that the representation dimension can be arbitrary large.

Usually it is difficult to find a precise generator of a given triangulated category with finite dimension; see \cite{IT2016}. 
Due to Keller, Murfet, and Van den Bergh \cite{KMVdB}, for an isolated singularity $(R,\m,k)$, the singularity category of $R$ is generated by $k$; we recover this result in Corollary \ref{iso}. Inspired by this result and  Theorem \ref{locus}, we give an upper bound for the dimension of the singularity category of an equicharacteristic excellent local ring with isolated singularity. 
\begin{theorem}\label{t3} (see \ref{upper bound})
Let $(R,\m, k)$ be an equicharacteristic excellent local ring. If $R$ has an isolated singularity, then

(1) $\ann_R \D_{\sg}(R)$ is $\m$-primary.

(2) For any $\m$-primary ideal $I$ that is contained in $\ann_R\D_{\sg}(R)$,  then  
$k$ is a generator of $\D_{\sg}(R)$ with generation time at most $(\nu(I)-\depth(R)+1)\ell\ell(R/I)$. 
\end{theorem}

In the above result, $\nu(I)$ is the minimal number of generators of $I$ and $\ell\ell(R/I)$ is the Loewy length of $R/I$, i.e. the minimal integer $n\in \N$ such that $(\m/I)^n=0$.

Theorem \ref{t3} builds on ideas from a result of Dao and Takahashi \cite{DT2015} and extends their result to non Cohen-Macaulay rings; see Remark \ref{connection}. The key new ingredient in our proof makes use of Theorem \ref{locus}.

\begin{ack}
This work was inspired by the collaboration with Srikanth Iyengar, Janina Letz, and Josh Pollitz \cite{ILLP}. During the collaboration with them,  the author learned about the annihilator of the singularity category. The author would like to thank them for their discussions and valuable comments. During this work,  the author visited China Jiliang University and Northeast Normal University. The author would like to thank Xianhui Fu, Pengjie Jiao, and Junling Zheng for their hospitality and discussions.
\end{ack}

\section{Notation and Terminology}
Throughout this article, $R$ will be a commutative noehterian ring. 
\begin{chunk}
\textbf{Derived category and singularity category.} Let $\D(R)$ denote the derived category of $R$-modules. It is a triangulated category with the shift functor $\Sigma$; for each complex $X\in \D(R)$, $\Sigma(X)$ is given by $\Sigma(X)^i=X^{i+1}$ and $\del_{\Sigma(X)}=-\del_X$. 

We let $\D^f(R)$ denote the full subcategory of $\D(R)$ consisting of  complexes $X$ such that the total cohomology $\bigoplus_{i\in \Z}\h^i(X)$ is a finitely generated $R$-module. $\D^f(R)$ inherits the structure of triangulated category from $\D(R)$.

A complex $X\in \D^f(R)$ is called \emph{perfect} if it is isomorphic to bounded complex of finitely generated projective $R$-modules. We let $\per(R)$ denote the full subcategory of $\D^f(R)$ consisting of perfect complexes. The \emph{singularity category} of $R$ is the Verdier quotient  $$\D_{\sg}(R)\colonequals\D^f(R)/\per(R).$$
This was first introduced by Buchweitz \cite[Definition 1.2.2]{Buchweitz1987} under the name ``stable derived category"; see also \cite{Orlov04}. For two complexes $X,Y\in \D_{\sg}(R)$, recall that each morphism from $X$ to $Y$ in $\D_{\sg}(R)$ is of the form $X\xleftarrow \alpha Z\xrightarrow \beta Y$, where $\alpha,\beta$ are morphisms in $\D^f(R)$ and the cone of $\alpha$ is a perfect complex; see \cite{Verdier}.
\end{chunk}

\begin{chunk}
\textbf{Thick subcategory.} Let $\T$ be a  triangulated category. A subcategory $\CC$ of $\T$ is called \emph{thick} if $\CC$ is closed under shifts, cones, and direct summands. For example, $\per(R)$ is a thick subcategory of $\D^f(R)$; see \cite[Lemma 1.2.1]{Buchweitz1987}.

For each object $X$ in $\T$, set $\thick_{\T}^0(X)=\{0\}$. Denote by $\thick_\T^1(X)$ the smallest full subcategory of $\T$ that contains $X$ and is closed under finite direct sums, direct summands, and shifts. Inductively, let $\thick_\T^n(X)$ denote the full subcategory of $\T$ consisting of objects $Y\in\T$ that fit into an exact triangle
$$
Y_1\rightarrow Y\oplus Y^\prime\rightarrow Y_2\rightarrow \Sigma(Y_1),
$$
where $Y_1\in \thick^1_{\T}(X)$ and $Y_2\in \thick_{\T}^{n-1}(X)$.
Note that the smallest thick subcategory of $\T$ containing $X$, denoted $\thick_\T(X)$, is precisely $\bigcup_{n\geq 0}\thick_\T^n(X)$. 
\end{chunk}
\begin{chunk}\label{def of dim}
\textbf{Dimension of triangulated categories. } Let $\T$ be a triangulated category. The \emph{dimension} of $\T$ introduced by Rouquier \cite{Rouquier} is defined to be 
$$
\dim \T\colonequals\inf\{n\in \N\mid \text{ there exists }G\in \T \text{ such that }\T=\thick_\T^{n+1}(G)\}.
$$
Let $G$ be an object in $\T$. $G$ is called a \emph{generator} of $\T$ if $\thick_\T(G)=\T$. $G$ is called a \emph{strong generator} of $\T$ if $\thick_\T^n(G)=\T$ for some $n\in \N$. The minimal number $n$ such that $\thick_\T^n(G)=\T$ is called the \emph{generation time} of $G$ in $\T$.

For example, if $R$ is an artinian ring, then $R/J(R)$ is a strong generator of $\D^f(R)$ with generation time at most $\ell\ell(R)$, where $J(R)$ is the Jacobian radical of $R$ and $\ell\ell(R)\colonequals\inf\{n\in \N\mid J(R)^n=0\}$ is the Loewy length of $R$; see \cite[Proposition 7.37]{Rouquier}. 
\end{chunk}

\begin{chunk}\label{def of syzygy}
\textbf{Syzygy modules.} For a finitely generated $R$-module $M$ and $n\geq 1$, we let $\Omega_R^n(M)$ denote the $n$-th syzygy of $M$. That is, there is a long exact sequence
$$
0\rightarrow \Omega_R^n(M)\rightarrow P^{-(n-1)}\rightarrow \cdots P^{-1}\rightarrow P^0\rightarrow M\rightarrow 0,
$$
where $P^{-i}$ are finitely generated projective $R$-modules for all  $0\leq i\leq n-1$. By Schanuel's Lemma, $\Omega^n_R(M)$ is independent of the choice of the projective resolution of $M$ up to projective summands. 

When $R$ is local, by choosing the minimal free resolution of $M$, the module $\Omega_R^n(M)$ has no projective summands. In this case, we always assume $\Omega_R^n(M)$ has no projective summands in the article. 
\end{chunk}
\begin{chunk}
\textbf{Support of modules.} Let $\Spec(R)$ denote the set of all prime ideals of $R$. It is endowed with the Zariski topology. A closed subset in this topology is of the form $V(I)\colonequals\{\p\in \Spec(R)\mid \p\supseteq I\}$, where $I$ is an ideal of $R$. For each $R$-module $M$, the \emph{support} of $M$ is 
$$
\Supp_RM\colonequals\{\p\in \Spec(R)\mid M_\p\neq 0\},
$$
where $M_\p$ is the localization of $M$ at $\p$.
\end{chunk}

 \section{Localization and annihilator of triangulated categories}\label{section2}
Throughout this section, $R$ will be a  commutative noetherian ring and $\T$ will be an essentially small $R$-linear triangulated category. 
\begin{chunk}
 We say the triangulated category $\T$ is $R$-\emph{linear} if for each $X\in \T$, there is a ring homomorphism $$\phi_X\colon R\rightarrow \Hom_{\T}(X,X)$$ such that the $R$-action on $\Hom_{\T}(X,Y)$ from the right via $\phi_X$ and from the left via $\phi_Y$ are compatible.  That is, for each $r\in R$ and $\alpha\in \Hom_{\T}(X,Y)$, one has $$\phi_Y(r)\circ \alpha=\alpha\circ \phi_X(r).$$
\end{chunk}

\begin{chunk}\label{def}
For each $X\in \T$, the \emph{annihilator} of $X$, denoted $\ann_RX$, is defined to be the annihilator of $\Hom_{\T}(X,X)$ over $R$. That is,
$$
\ann_RX\colonequals\{r\in R\mid r\cdot \Hom_\T(X,X)=0\}.
$$
The annihilator of $\T$ is defined to be 
$$
\ann_R\T\colonequals\bigcap_{X\in \T}\ann_RX.
$$
\end{chunk}
A commutative noetherian local ring is called \emph{regular} if its maximal ideal can be generated by a system of parameter. Due to Auslander, Buchsbaum, and Serre, a commutative noetherian local ring is regular if and only if its global dimension is finite; see \cite[Theorem 2.2.7]{BH}. A commutative noetherian ring $R$ is called regular provided that $R_\p$ is regular for all $\p\in \Spec(R)$. 

\begin{example}\label{regular}
Consider the $R$-linear triangulated category $\D_{\sg}(R)$. As mentioned in the introduction,  $R$ is regular if and only if  $\ann_R\D_{\sg}(R)=R$. Indeed, it is clear that $\ann_R\D_{\sg}(R)=R$ ( $ \iff\D_{\sg}(R)$ is trivial) is equivalent to that  every finitely generated $R$-module has finite projective dimension. It turns out that this is equivalent to $R$ is regular. According to Auslander, Buchsbaum, and Serre's criterion, the forward direction is clear. 
For the backward direction, see \cite[Lemma 4.5]{BM}.
\end{example}
\begin{chunk}
Let $V$ be a \emph{specicalization closed subset} of $\Spec(R)$; that is, if $\p\in V$, then the prime ideal $\q$ is in $V$ if $\p\subseteq \q$. Following Benson, Iyengar, and Krause \cite[Section 3]{BIK2015}, we define $\T_V$ to be full subcategory
$$
\T_V\colonequals\{X\in \T\mid \Hom_{\T}(X,X)_\p=0 \text{ for all }\p\in \Spec(R)\setminus V\}.
$$
We observe that $\T_V$ is a thick subcategory of $\T$ as the $R$-action on $\Hom_{\T}(X,Y)$ factors through $\End_{\T}(X)$-action on $\Hom_\T(X,Y)$ and $\End_\T(Y)$-action on $\Hom_\T(X,Y)$.

 For each prime ideal $\p$ of $R$,  set 
 $$Z(\p)\colonequals\{\q\in \Spec(R)\mid \q\nsubseteq \p\}.$$
 Then $Z(\p)$ is a specialization closed subset of $\Spec(R)$.
 The \emph{localization} of $\T$ at $\p$ is defined to be the verdier quotient
 $$
 \T_\p\colonequals\T/ \T_{Z(\p)}.
 $$
 \end{chunk}
\begin{example}
Consider the $R$-linear triangulated category $\D^f(R)$. Since $R$ is noetherian, for $X,Y\in \D^f(R)$, one has
$$
\Hom_{\D^f(R)}(X,Y)_\p\cong \Hom_{\D^f(R_\p)}(X_\p,Y_\p).
$$
This immediately yields that $\Hom_{\D^f(R)}(X,X)_\p=0$ if and only if $X_\p=0$ in $\D^f(R_\p)$; the latter means $X_\p$ is acyclic. We conclude that
$$
\D^f(R)_{Z(\p)}=\{X\in \D^f(R)\mid X_\p \text{ is acyclic}\}.
$$
Combining with this,  \cite[Lemma 2.2]{Orlov2011} implies that $\D^f(R)/\D^f(R)_{Z(\p)}\cong \D^f(R_\p)$. That is, there is a triangle equivalence $$\D^f(R)_\p\cong \D^f(R_\p).$$
\end{example}
We will show that an analogue of the above example holds for the singularity category; see Corollary \ref{sin}.

\begin{lemma}\label{support}
For each object $X$ in  $\T$, we have
$$
\Supp_R \Hom_{\T}(X,X)=V(\ann_{R}X).
$$
In particular, $\Supp_R\Hom_\T(X,X)$ is a closed subset of $\Spec(R)$. 
\end{lemma}
\begin{proof}
The second statement follows immediately from the first one.

It is clear $\Supp_R\Hom_{\T}(X,X)\subseteq V(\ann_{R}X)$.
For the converse, let $\ann_{R}X\subseteq \p$ for some prime ideal $\p$ of $R$. We claim that $\Hom_{\T}(X,X)_\p\neq 0$. If not, assume $\Hom_{\T}(X,X)_\p=0$. Consider the identity morphism $\id_X\colon X\rightarrow X$ in $\Hom_\T(X,X)$. The assumption yields that $\id_X$ is zero in the localization $\Hom_\T(X,X)_\p$. Thus there exists $r\notin \p$ such that $r\cdot \id_X=0$. Then it is clear that $r\in \ann_{R}X$. Hence $\ann_RX\nsubseteq \p$. This contradicts with $\ann_RX\subseteq \p$.  As required.
\end{proof}

\begin{chunk}\label{def of Kos}
Let $X$ be an object in $\T$. Given an element $r\in R$, the \emph{Koszul object} of $r$ on $X$, denoted $X\para r$ , is the object that fits into the exact triangle
$$
X\xrightarrow r X\rightarrow X\para r \rightarrow \Sigma(X).
$$
That is, $X\para r$ is the cone of the map $r\colon X\rightarrow X$.  For a sequence $\bm{r}=r_1,\ldots,r_n$, one can define the Koszul object $X\para\bm{r}$ by induction on $n$. It is not difficult to show 
\begin{equation}\label{Koszul}
    \Supp_R \Hom_\T(X\para \bm{r},X\para \bm{r})\subseteq \Supp_R\Hom_\T(X,X)\cap V(\bm{r}).
\end{equation}

\end{chunk}

 The following result is a direct consequence of \cite[Lemma 3.5]{BIK2015}. 
\begin{lemma}\label{structure}
For each prime ideal $\p$ of $R$,  
$$
\T_{Z(\p)}=\thick_{\T}(X\para r\mid X\in \T, r\notin \p)
$$
and the quotient functor $\T\rightarrow \T/\T_{Z(\p)}=\T_\p$ induces a natural isomorphism
$$
\Hom_{\T}(X,Y)_\p\cong \Hom_{\T_\p}(X,Y)
$$
for $X,Y$ in $\T$. \qed
\end{lemma}




\begin{corollary}\label{module}
Let $X$ be an object in $\T$. Then
$$
\{\p\in \Spec(R)\mid X\neq 0 \text{ in } \T_\p\}=V(\ann_RX).
$$
\end{corollary}
\begin{proof}
By Lemma \ref{support},
$
V(\ann_{R}X)=\Supp_R \Hom_{\T}(X,X).
$
Note that the isomorphism
$
\Hom_{\T}(X,X)_{\p}\cong \Hom_{\T_\p}(X,X)
$
in Lemma \ref{structure}
yields that $\Hom_{\T}(X,X)_\p\neq 0$ is equivalent to $X\neq 0$ in $\T_\p$. This completes the proof.
\end{proof}

\begin{lemma}\label{basic result}
$\{\p\in \Spec(R)\mid \T_\p\neq 0\}\subseteq V(\ann_R \T).$
\end{lemma}
\begin{proof}
By definition $\ann_R\T\subseteq \ann_RX$ for each $X\in \T$. Thus we get that
$
V(\ann_RX)\subseteq V(\ann_R\T).
$ 
Combining with Corollary \ref{module}, we get
$$
\{\p\in \Spec(R)\mid \T_\p\neq 0\}=\bigcup_{X\in\T}\{\p\in \Spec(R)\mid X\neq 0 \text{ in }\T_\p\}\subseteq V(\ann_R\T).
$$
As required.
\end{proof}

The following is the main result of this section.
\begin{proposition}\label{theorem}
Let $\T$ be an essentially small $R$-linear triangulated category. If 
$\dim \T<\infty$, then
$$
\{\p\in \Spec R\mid \T_\p\neq 0\}= V(\ann_R\T).
$$
\end{proposition}
\begin{proof}
Assume $\T=\thick^n_{\T}(G)$ for some $G\in \T$ and $n\in \N$. Set $I:=\ann_{R}G$. Then $I^n\subseteq \ann_R\T$; see \cite[Lemma 2.1]{Esentepe}. In particular, $\V(\ann_R \T)\subseteq V(I)$.

We claim that $V(I)\subseteq \{\p\in \Spec(R)\mid \T_\p\neq 0\}$. Indeed, let $\p\in \Spec (R)$ and $I\subseteq \p$, by Lemma \ref{support} we have $\Hom_{\T}(G,G)_\p\neq 0$. Thus we conclude that $\T_\p\neq 0$ by Lemma \ref{structure}. 

By above, we have $V(\ann_R\T)\subseteq \{\p\in \Spec(R)\mid \T_\p\neq 0\}$. The desired result now follows immediately from Lemma \ref{basic result}.
\end{proof}

\section{Annihilators of the singularity category}
In this section, We will investigate the annihilator of $\D_{\sg}(R)$ over $R$. It turns out that the Jacobian ideal and the annihilator of $\D_{\sg}(R)$ are equal up to radical under some assumptions; see Corollary \ref{main result}.

First we give a technical lemma which is used in the proof of Lemma \ref{isomorphism} and Lemma \ref{loc}; the proof is inspired by \cite[Lemma 2.2]{HP}.
\begin{lemma}\label{technique}
Let $X$ be an object in $\D_{\sg}(R)$ and $\p$ be a prime ideal of $R$. If $X_\p$ is perfect over $R_\p$, then there exists $r\notin \p$ such that
$X$ is a direct summand of $\Sigma^{-1}(X\para r)$ in $\D_{\sg}(R)$.
\end{lemma}
\begin{proof}
By choosing a projective resolution of $X$, we may assume $X$ is a bounded above complex of finitely generated projective $R$-modules with finitely many non-zero cohomologies. Then by taking brutal truncation, we conclude that $\Sigma^n(X)$ is isomorphic to a finitely generated $R$-module in $\D_{\sg}(R)$ for $ n\ll 0$. Combining with the assumption, we may assume $X$ is a finitely generated $R$-module and $X_\p$ is a free $R_\p$-module. 

Choose a projective resolution $\pi\colon P(X)\rightarrow X$, where $P(X)$ is a finitely generated projective $R$-module. The kernel of $\pi$ is the first syzygy of $X$, denoted $\Omega^1_R(X)$. Then we have $\Ext_R^1(X,\Omega_R^1(X))_\p=0$ as $X_\p$ is a free $R_\p$-module.
Since $\Ext^1_R(X,\Omega_R^1(X))$ is finitely generated over $R$, there is an element $r\notin \p$ such that $r\cdot\Ext_R^1(X,\Omega^1_R(X))=0$. That is, there exists a commutative diagram
\begin{equation}\label{diagram}
\xymatrix{
0\ar[r]&\Omega_R^1(X)\ar@{=}[d]\ar[r]^-{\left(\begin{smallmatrix}
0\\
1
\end{smallmatrix}\right)}& X\oplus \Omega_R^1(X)\ar[r]^-{(1,0)}\ar[d]& X\ar[r]\ar[d]^-r& 0\\
0\ar[r]& \Omega_R^1(X)\ar[r]  & P(X)\ar[r]^-{\pi}& X\ar[r]& 0
}
\end{equation}
in the category of $R$-modules.

Let $f$ denote the middle map $X\oplus \Omega_R^1(X)\rightarrow P(X)$ in (\ref{diagram}). The right square of (\ref{diagram}) induces a morphism $\varphi\colon \cone(f)\rightarrow X\para r$, where $\cone(f)$ is the cone of $f$.
It follows immediately from the snake lemma that $\varphi$ is a quasi-isomorphism. Hence there exists an exact triangle 
$$
X\oplus \Omega_R^1(X)\rightarrow P(X)\rightarrow X\para r \rightarrow \Sigma(X\oplus \Omega_R^1(X))
$$
in $\D^f(R)$. Thus in $\D_{\sg}(R)$, we get that $X\para r\cong \Sigma(X\oplus \Omega_R^1(X))$.  As required.
\end{proof}
\begin{lemma}\label{isomorphism}
$\D_{\sg}(R)_\p= \D_{\sg}(R)/\{X\in \D_{\sg}(R)\mid X_\p=0\in \D_{\sg}(R_\p)\}$. 
\end{lemma}
\begin{proof} 
It is equivalent to show 
$$
\D_{\sg}(R)_{Z(\p)}=\{X\in \D_{\sg}(R)\mid X_\p=0 \in \D_{\sg}(R_\p)\}.
$$
From Lemma \ref{structure},
$
\D_{\sg}(R)_{Z(\p)}=\thick_{\D_{\sg}(R)}(X\para r\mid X\in \D_{\sg}(R), r\notin \p). 
$
For each $X\in \D_{\sg}(R)$ and $r\notin \p$, since $(X\para r)_\p$ is acyclic, we conclude that $\D_{\sg}(R)_{Z(\p)}\subseteq \{X\in \D_{\sg}(R)\mid X_\p=0\in \D_{\sg}(R_\p)\}.$

For the reverse inclusion, assume $X\in \D_{\sg}(R)$ and $X_\p=0$ in $\D_{\sg}(R_\p)$. Lemma \ref{technique} yields that $X\in \thick_{\D_{\sg}(R)}(X\para r)$ for some $r\notin \p$.  This completes the proof.
\end{proof}

\begin{lemma}\label{loc}
Let $R$ be a commutative noetherian ring. For objects $X,Y$ in $\D_{\sg}(R)$, there is a natural isomorphism
$$
\Hom_{\D_{\sg}(R)}(X,Y)_\p\cong \Hom_{\D_{\sg}(R_\p)}(X_\p,Y_\p).
$$
\end{lemma}

\begin{proof}
We define the map $\pi\colon \Hom_{\D_{\sg}(R)}(X,Y)_\p\rightarrow \Hom_{\D_{\sg}(R_\p)}(X_\p,Y_\p)$ by sending $s^{-1}(\alpha/\beta)$ to $X_\p\xleftarrow{s\circ \beta_\p} Z_\p\xrightarrow{\alpha_\p} Y_\p$, where $s\notin\p$ and $\alpha/\beta$ is $X\xleftarrow \beta Z\xrightarrow \alpha Y$; here $\alpha,\beta$ are morphisms in $\D^f(R)$ and $\cone(\beta)$ is perfect over $R$.
The map is well-defined.

First we prove the map is injective. If $\pi(s^{-1}(\alpha/\beta))=0$, then $\alpha_\p$ factors through a perfect complex over $R_\p$. With the same argument in the proof of \cite[Lemma 3.9]{Letz}, one can verify that $(-)_\p\colon \per(R)\rightarrow \per(R_\p)$ is dense. Hence $\alpha_\p$ factors through $F_\p$, where $F\in \per(R)$. Since for $M,N\in \D^f(R)$
$$
\Hom_{\D^f(R)}(M,N)_\p\cong \Hom_{\D^f(R_\p)}(M_\p,N_\p),
$$
there exists $\gamma\colon Z\rightarrow F$ and $\eta \colon F\rightarrow Y$ in $\D^f(R)$ such that $\alpha_\p=t_1^{-1}\eta_\p\circ t_2^{-1}\gamma_\p$ for some $t_1,t_2\notin \p$. This implies that there exists $t\notin \p$ such that $tt_1t_2\alpha=t\eta\circ\gamma$. Since $tt_1t_2\notin\p$, we get that $s^{-1}(\alpha/\beta)=0$. Thus $\pi$ is injective.

Now we prove that the map is surjective. We just need to consider the map $X_\p\xleftarrow {g_\p} W_\p\xrightarrow {f_\p} Y_\p$ is in the image of $\pi$ for each $W\in \D^f(R)$, where $f\colon W\rightarrow Y$ in $\D^f(R)$, $g\colon W\rightarrow X$ in $\D^f(R)$, and $\cone(g)_\p$ is perfect over $R_\p$. Then Lemma \ref{technique} yields that $\cone(g)$ is a direct summand of $\Sigma^{-1}(\cone(g)\para r)$ in $\D_{\sg}(R)$ for some $r\notin \p$. Since the multiplication $r\colon \cone(g)\para r\rightarrow \cone(g)\para r$ is null-homotopy, $r/1\colon \cone(g)\para r\rightarrow \cone(g)\para r$ is zero in $\D_{\sg}(R)$. Hence $r/1\colon \cone(g)\rightarrow \cone(g)$ is also zero in $\D_{\sg}(R)$. Combining with the exact triangle $W\xrightarrow {g/1} X\rightarrow \cone(g)\rightarrow \Sigma(W)$ in $\D_{\sg}(R)$,
 we conclude that $r/1\colon X\rightarrow X$ factors through $g/1$ in $\D_{\sg}(R)$. Assume $r/1=g/1\circ h_1/h_2$, where $h_1/h_2$ is $X\xleftarrow {h_2} L \xrightarrow{h_1} W$ and $\cone(h_2)$ is perfect over $R$. This implies $r/1=(g\circ h_1)/h_2 $. Hence there exists a commutative diagram in $\D^f(R)$
 $$
 \xymatrix{
 & L\ar[rd]^-{g\circ h_1}\ar[ld]_-{h_2}& \\
 X& L^\prime\ar[r]^{rl}\ar[l]_-l\ar[u]_-{h_3}\ar[d]^l& X,\\
 & X\ar[ru]_-r\ar[lu]^-{1}& 
 }
 $$
 where $\cone(l)$ is perfect over $R$. Note that $g\circ h_1\circ h_3=rl$. As $\cone((rl)_\p)$ is perfect over $R_\p$, we get that $f_\p/g_\p=(f\circ h_1\circ h_3)_\p/(rl)_\p$. This morphism is precisely $\pi(r^{-1}(f\circ h_1\circ h_3/l))$. This completes the proof.
\end{proof}

\begin{corollary}\label{sin}
For a commutative noetherian ring $R$,  we have
$$
\D_{\sg}(R)_\p= \D_{\sg}(R)/\{X\mid X_\p=0 \text{ in } \D_{\sg}(R_\p)\}\cong \D_{\sg}(R_\p).
$$
\end{corollary}
\begin{proof}
The first equation is from Lemma \ref{isomorphism}. Combining with this, the localization functor $\D_{\sg}(R)\rightarrow \D_{\sg}(R_\p)$ induces a triangle functor
$\pi\colon\D_{\sg}(R)_\p\rightarrow \D_{\sg}(R_\p)$. $\pi$ is fully faithful by Lemma \ref{structure} and Lemma \ref{loc}. By \cite[Lemma  3.9]{Letz}, $\pi$ is dense. Thus $\pi$ is an equivalence.
\end{proof}
\begin{remark}
(1) When $R$ is a Gorenstein local ring, the second equivalence above was proved by Matsui \cite[Lemma 4.12]{Matsui} using a different method. 

(2) 
 Let $X$ be a finitely generated $R$-module. Since $\pd_R(X)<\infty$ if and only if $X=0$ in $\D_{\sg}(R)$, Corollary \ref{module} and Corollary \ref{sin} yield that 
 $$
 \{\p\in \Spec(R)\mid \pd_{R_\p}(X_\p)=\infty\}=V(\ann_R\Hom_{\D_{\sg}(R)}(X,X)).
 $$
 In particular, the set $\{\p\in\Spec(R)\mid \pd_{R_\p}(M_\p)<\infty\}$ is Zariski open; this is proved in \cite[Lemma 4.5]{BM}.
\end{remark}

Let $\Sing(R)$ denote the \emph{singular locus} of $R$. That is, $\Sing(R)\colonequals\{\p\in \Spec(R)\mid R_\p\text{ is not regular}\}.$

\begin{theorem}\label{locus}
Let $R$ be a commutative noetherian ring. If $\dim\D_{\sg}(R)<\infty$, then
$$
\Sing (R)=V(\ann_R\D_{\sg}(R)).
$$
In particular, in this case $\Sing(R)$ is a closed subset.
\end{theorem}
\begin{proof}
For each prime ideal $\p$ of $R$, by Corollary \ref{sin} we get that $\D_{\sg}(R)_\p\neq 0$ if and only if $\D_{\sg}(R_\p)\neq 0$. This is equivalent to $\p\in \Sing(R)$. Thus the desired result follows immediately from Proposition \ref{theorem}.
\end{proof}

\begin{remark}\label{finite}
Let $R$ be a localization of a finitely generated algebra over a field or an equicharacteristic excellent local ring. 
It is proved by Iyengar and Takahashi that $\dim \D^f(R)<\infty$; see \cite[Corollary 7.2]{IT2016}. In particular,  $\dim \D_{\sg}(R)<\infty$. 

In this case, Iyengar and Takahashi \cite[5.3 and 5.4]{IT2016}  proved that the cohomological annihilator  (see \ref{coh}), denoted $\ca(R)$, defines the singular locus of $R$. Combining with Theorem \ref{locus}, we conclude that $\ca(R)$ is equal to $\ann_R \D_{\sg}(R)$ up to radical. We will give a more precise relation between them in Proposition \ref{relation}.

\end{remark}

\begin{chunk}\label{jacobian}
Let $R$ be a finitely generated algebra over a field $k$ (resp.an equicharacteristic complete local ring).
Then $R\cong k[x_1,\ldots,x_n]/(f_1,\ldots,f_c)$ (resp. $R\cong  k\llbracket x_1,\ldots,x_n\rrbracket/(f_1,\ldots,f_c)$ by Cohen's structure theorem, where $k$ is the residue field of $R$). Denote by $h$ the height of the ideal $(f_1,\ldots,f_c)$ in $k[x_1,\ldots,x_n]$ (resp. $k\llbracket x_1,\ldots,x_n\rrbracket$). More precisely, $h=n-\dim (R)$; see \cite[Theorem I 1.8A]{Hartshorne} (resp.  \cite[Corollary 2.1.4]{BH}). The \emph{Jacobian ideal} of $R$, denoted $\jac(R)$,  is defined to be the ideal of $R$ generated by all $h\times h$ minors of the Jacobian matrix $$\del(f_1,\ldots,f_c)/\del(x_1,\ldots,x_n).$$
\end{chunk}

Recall that a commutative noetherian ring is called \emph{equidimensional} provided that $\dim R/\p=\dim R/\q<\infty$  for all minimal prime ideals $\p,\q$ of $R$.
\begin{corollary}\label{main result}
Let $R$ be either an equidimensional finitely generated $k$-algebra over a perfect field $k$, or an equidimensional equicharacteristic complete local ring with a perfect residue field. Then
 $$
 \sqrt{\jac(R)}=\sqrt{\ann_R\D_{\sg}(R)}.
 $$
 In particular, $\jac(R)^s$ annihilates the singularity category of $R$ for some integer $s$.
\end{corollary}
\begin{proof}
The last statement follows immediately from the first one.

In both cases, $\jac(R)$ defines the singular locus of $R$. That is,
$$
\Sing(R)=V(\jac(R)).
$$
Indeed, the affine case can see \cite[Corollary 16.20]{Ei}. The local case can combine \cite[Lemma 2.10]{IT2016} and \cite[Proposition 4.4, Proposition 4.5, and Theorem 5.4]{Wang1994}.

From Remark \ref{finite}, 
$
\dim \D_{\sg}(R)<\infty.
$
Combining with this,  Theorem \ref{locus} implies that
$$
\Sing(R)=V(\ann_R\D_{\sg}(R)).
$$
By above two equations, we have
$$
V(\jac(R))=V(\ann_R\D_{\sg}(R)).
$$
This implies the desired result. 
\end{proof}
\begin{remark}
(1) When $R$ is an equicharacteristic Cohen-Macaulay local ring over a field, it turns out that $\jac(R)$ annihilates the singularity category of $R$; see \cite{IT2021}.

(2) Corollary \ref{main result}  fails without equidimensional assumption; see Example \ref{fail}. The example also shows that the power of the Jocobian ideal doesn't annihilate the singularity category without equidimensional assumption.
\end{remark}
\begin{example}\label{fail}
Let $R=k[x,y,z,w]/(x^2,yz,yw)$ (resp. $k\llbracket x,y,z,w\rrbracket/(x^2,yz,yw)$), where $k$ is a field with characteristic $0$. This is not equidimensional.  
Consider the prime ideal $\p=(\overline{x},\overline{z},\overline{w})$ of $R$. Note that $R_\p$ is not regular.  
Thus by Lemma \ref{basic result} and Corollary \ref{sin}, we get that
\begin{equation}\label{345}
    \p\in \Sing(R)\subseteq V(\ann_R\D_{\sg}(R)).
\end{equation}
In particular, $\ann_R\D_{\sg}(R)\subseteq \p$.

The height of $(x^2,yz,yw)$ in $k[x,y,z,w]$ (resp. $k\llbracket x,y,z,w\rrbracket$) is $2$. Then it is easy to compute that
$$\jac(R)=(\overline{xy},\overline{xz},\overline{xw},\overline{y^2}).$$
Combining (\ref{345}) with $\jac(R)\nsubseteq \p$, we conclude that
$$
\jac(R)\nsubseteq \sqrt{\ann_R\D_{\sg}(R)}.
$$
\end{example}

\section{Comparison with the cohomological annihilator}\label{section4}
In this section, we compare the annihilator of the singularity category with the cohomological annihilator. The main result of this section is Proposition \ref{t2} from the introduction. Using this result, we calculate an example of the annihilator of the singularity category at the end of this section.

\begin{chunk}\label{coh}
For each $n\in \N$, following Iyengar and Takahashi \cite[Definition 2.1]{IT2016}, the $n$-th \emph{cohomological annihilator} of $R$ is defined to be
$$
\ca^n(R):=\ann_R \Ext_R^{ n}(R\text{-}\mo,R\text{-}\mo),
$$
where $R\text{-}\mo$ is the category of finitely generated $R$-modules.
In words, $\ca^n(R)$ consists of elements $r$ in $R$ such that $r\cdot \Ext^n_R(M,N)=0$ for all finitely generated $R$-modules $M,N$. 
The \emph{cohomological annihilator} of $R$ is defined to be 
$$
\ca(R)\colonequals\bigcup_{n\geq 0}\ca^n(R).
$$
It is proved that $\ca^n(R)$ is equal to the ideal $\ann_R\Ext^{\geq n}_R(R\text{-}\mo,R\text{-}\mo)$. In particular, there is an ascending chain of ideals  $0=\ca^0(R)\subseteq \ca^1(R)\subseteq \ca^2(R)\subseteq\cdots$. As $R$ is noehterian, there exists $N\in \N$ such that $\ca(R)=\ca^n(R)$ for all $n\geq N$. 

It is not difficult to verify that there is an inclusion
$$
   \Sing(R)\subseteq V(\ca(R)); 
$$
see \cite[Lemma 2.10]{IT2016}.
\end{chunk}

\begin{chunk}
Let $R$ be a strongly Gorenstein ring, i.e. $R$ has finite injective dimension as $R$-module.  It is proved by Esentepe \cite[Lemma 2.3]{Esentepe} that in this case
$$
\ca(R)=\ann_R\D_{\sg}(R).
$$
Combining with this result, if furthermore $\dim \D_{\sg}(R)<\infty$, then Theorem \ref{locus} yields that 
\begin{equation}\label{123}
    \Sing(R)=V(\ca(R)).
\end{equation}

When $R$ is a Gorenstein local ring and $\dim\D_{\sg}(R)<\infty$, (\ref{123}) was proved by Bahlekeh, Hakimian, Salarian, and Takahashi  \cite[Theorem 3.3]{BHST}.
\end{chunk}

It is natural to ask: what is the relation of $\ca(R)$ and $\ann_R\D_{\sg}(R)$ when $R$ is not Gorenstein? It turns out that they are equal up to radical under some mild assumptions.

\begin{proposition}\label{relation}
Let $R$ be a commutative noetherian ring. Then

(1) $\ca(R)\subseteq \ann_R \D_{\sg}(R)$.

 (2) If furthermore $R$ is either a localization of a finitely generated algebra over a field or an equicharacteristic excellent local ring, then 
 $$
\sqrt{\ca(R)}=\sqrt{ \ann_R\D_{\sg}(R)}.
$$
\end{proposition}
\begin{proof}
(1) It is equivalent to show that $\ca^n(R)\subseteq \ann_R \D_{\sg}(R)$ for all $n\geq 1 $. For each $r\in \ca^n(R)$ and $X\in \D_{\sg}(R)$, we want to show the multiplication $r\colon X\rightarrow X$ is zero in $\D_{\sg}(R)$. In order to prove this, we may assume $X\cong \Omega^{n-1}_R(Y)$ for some $R$-module $Y$, where $\Omega^{n-1}_R(Y)$ is a $(n-1)$-th syzygy of the $R$-module $Y$; see the argument in the proof of Lemma  \ref{technique}.

Choose a short exact sequence
$$
0\rightarrow \Omega_R^1(X)\rightarrow P(X)\xrightarrow \pi X\rightarrow 0,
$$
where $\pi$ is a projective resolution of $X$.
Note that 
$$\Ext^1_R(X,\Omega^1_R(X))\cong \Ext^{n}_R(Y,\Omega^1_R(X)).$$ 
Combining with the assumption, we get that $r\cdot \Ext^1_R(X,\Omega^1_R(X))=0$. This will imply that $r\colon X\rightarrow X$ factors through the morphism $\pi$. In particular, $r\colon X\rightarrow X$ factors through the projective module $P(X)$. Thus $r\colon X\rightarrow X$ is zero in $\D_{\sg}(R)$. As required. 

(2) By Theorem \ref{locus} and Remark \ref{finite}, we have
$$
V(\ca(R))=\Sing(R)=V(\ann_R\D_{\sg}(R)).
$$
This yields 
$
\sqrt{\ca(R)}=\sqrt{\ann_R\D_{\sg}(R)}.
$
\end{proof}
\begin{corollary}\label{socle}
Let $R$ be a commutative noetherian local ring. Then the socle of $R$ annihilates the singularity category of $R$.
\end{corollary}
\begin{proof}
It is proved that the cohomological annihilator contains the socle of $R$; see \cite[Example 2.6]{IT2016}. The desired result follows immediately from Proposition \ref{relation}.
\end{proof}
\begin{example}
Let $R=k[x,y,z,w]/(x^2,yz,yw)$ (resp. $k\llbracket x,y,z,w\rrbracket/(x^2,yz,yw)$), where $k$ is a field with characteristic $0$. This is not equidimensional.  Combining Example \ref{fail} with Proposition \ref{relation}, we conclude that
$$
\jac(R)\nsubseteq \sqrt{\ca(R)}=\sqrt{\ann_R\D_{\sg}(R)}.
$$
\end{example}
\begin{remark}
The above example also shows that \cite[Theorem 1.1]{IT2021} need not hold without the equidimensional assumption.
\end{remark}

At the end of this section, we calculate an example of the annihilator of the singularity category. The ring considered in the following is not Cohen-Macaulay. 
\begin{example}\label{example}
Let $R=k\llbracket x,y\rrbracket /(x^2,xy)$, where $k$ is a field. We show 
$$
\jac(R)=\ca(R)=\ann_R \D_{\sg}(R)=(\overline{x},\overline{y}).
$$

First $\jac(R)=(\overline{x},\overline{y})$ is clear. By Example \ref{regular} and Proposition \ref{relation}, the desired result follows from $\ca(R)=(\overline{x},\overline{y})$. Since $\overline{x}$ lies in the socle of $R$, Remark \ref{socle} yields that $\overline{x}\in\ca(R) $. It remains to prove $\overline{y}\in \ca(R)$.
For any finitely generated $R$-module $M$, we claim
$
\overline{y}\cdot \Ext^3_R(M,-)=0.
$
This will imply $\overline{y}\in \ca^3(R)\subseteq \ca(R)$. 

Since there is an isomorphism $\Ext_R^3(M,-)\cong \Ext^2_R(\Omega_R^1(M),-)$, it is equivalent to show $\overline{y}\cdot \Ext_R^2(\Omega_R^1(M),-)=0$.
We observe $\overline{x}\cdot \Omega_R^1(M)=0$; see \ref{def of syzygy}. Thus $\Omega_R^1(M)$ is a finitely generated module over $R/(\overline{x})\cong k\llbracket y\rrbracket$.  It follows from the structure theorem of finitely generated modules over PID that $\Omega_R^1(M)$ is a finite direct sum of these modules: $R/(x), R/(x,y^n), n\geq 1$. Hence the claim follows if $\overline{y}\cdot\Ext^2_R(R/(x),-)=0=\overline{y}\cdot \Ext^2_R(R/(x,y^n),-)$ for all $n\geq 1$. The proof $\overline{y}\cdot\Ext^2_R(R/(x),-)=0$ is easier than $\overline{y}\cdot \Ext^2_R(R/(x,y^n),-)=0$. We prove the latter one for example.
The minimal free resolution of $R/(x,y^n)$ is
$$
\cdots \rightarrow R^5\xrightarrow{\begin{pmatrix}
x& y& 0& 0& 0\\
0& 0& x& 0& 0\\
0& 0& 0& x& y
\end{pmatrix}}R^3\xrightarrow{\begin{pmatrix}
x& y& 0\\
0& 0& x
\end{pmatrix}}R^2\xrightarrow{(x,y^n)}R\rightarrow 0.
$$
Hence for each $R$-module $N$, $\Ext_R^2(R/(x,y^n),N)$ is the second cohomology of 
$$
0\rightarrow N\xrightarrow{\begin{pmatrix}
x\\y^n
\end{pmatrix}} N^2\xrightarrow{\begin{pmatrix}
x& 0\\
y& 0\\
0& x
\end{pmatrix}} N^3\xrightarrow{\begin{pmatrix}x& 0&0\\
y& 0& 0\\
0& x& 0\\
0& 0& x\\
0& 0& y
\end{pmatrix}} N^5\rightarrow \cdots.
$$
If $(a,b,c)^T\in N^3$ is a cycle, then we get that $ya=yc=xb=0$. This implies
$$y\cdot\begin{pmatrix}
a\\
b\\
c
\end{pmatrix}=\begin{pmatrix}
0\\
yb\\
0
\end{pmatrix}=\begin{pmatrix}
x& 0\\
y& 0\\
0& x
\end{pmatrix}\begin{pmatrix}
b\\0
\end{pmatrix}.$$
In particular, $y\cdot(a,b,c)^T$ is a boundary.  Thus $\overline{y}\cdot\Ext^2_R(R/(x,y^n),N)=0 $. 
\end{example}

\section{Upper bound for dimensions of the singularity category}

The main result of this section is Theorem \ref{t3} from the introduction which gives an upper bound for the dimension of the singularity category of an equicharacteristic excellent local ring with isolated singularity. As mentioned in the introduction, it builds on ideas from  Dao and Takahashi's work \cite[Theorem 1.1(2) (a)]{DT2015} and extends their result; see  Remark \ref{connection}.


\begin{lemma}\label{theorem1}
 Let $(R,\m)$ be a commutative noehterian local ring and $\T$ be an essentially small $R$-linear triangulated category. Then the following are equivalent.

(1) $\{\p\in \Spec(R)\mid \T_\p\neq 0\}\subseteq\{\m\}.$

(2) For each $X\in \T$, there exists $j\in \N$ such that $\m^j\subseteq \ann_{R}X$.

(3) For each $X\in \T$, there exists an $\m$-primary ideal $(\bm{f})\colonequals(f_1,\ldots,f_l)$ such that $X\in \thick_{\T}(X\para \bm{f})$.
\end{lemma}
\begin{proof}
$(1)\Rightarrow (2)$: By Lemma \ref{support}, we get that for each $X\in \T$,
$$
V(\ann_{R}X)=\Supp_R \Hom_{\T}(X,X).
$$
The assumption implies that $\Supp_{R}\Hom_{\T}(X,X)\subseteq \{\m\}$. Thus  $V(\ann_{R}X)\subseteq \{\m\}$. This means $\m\subseteq \sqrt{\ann_{R}X}$. It follows that $\m^j\subseteq \ann_{R}X$ for some $j\in \N$.

$(2)\Rightarrow (3)$: By assumption, there exists $j\in \N$ such that $\m^j\subseteq \ann_{\T}X$. We write $\m^j=(\f)$, where $\f=f_1,\ldots,f_l$. Since $\m^j\subseteq \ann_{R}X$, $X$ is a direct summand of $X\para \f$ in $\T$. In particular, $X\in \thick_{\T}(X\para \bm{f})$.

$(3)\Rightarrow (1)$: We just need to show that for each $x\in \T$, $X$ is zero in $\T_\p$ if $\p\neq \m$. According to the hypothesis, it is enough to show $X\para\f=0$ in $\T_\p$ if $\p\neq \m$, where $(\f)$ is an $\m$-primary ideal. Combining with (\ref{Koszul}) in \ref{def of Kos}, we have
$$
\Supp_{R}\Hom_{\T}(X\para\f,X\para\f)\subseteq \{\m\}.
$$
The desired result follows immediately from Lemma \ref{structure}.
\end{proof}

Combining Corollary \ref{sin} with Lemma \ref{theorem1}, we recover the following result of Keller, Murfet, and Van den Bergh \cite[Proposition A.2]{KMVdB}.
\begin{corollary}\label{iso}
Let $(R,\m,k)$ be a commutative noetherian local ring. Then $R$ has an isolated singularity if and only if $\D_{\sg}(R)=\thick_{\D_{\sg}(R)}(k)$. \qed 
\end{corollary}

\begin{chunk}
For a commutative noetherian local ring $(R,\m, k)$ and a finitely generated $R$-module $M$, the \emph{depth} of $M$, denoted $\depth(M)$, is the length of a maximal $M$-regular sequence contained in $\m$. This is well defined as all maximal $M$-regular sequences contained in $\m$ have the same length; see \cite[Section 1.2]{BH} for more details.
\end{chunk}

\begin{lemma}\label{syzygy}
Let $(R,\m, k)$ be a commutative noetherian local ring and $X$ be a complex in $\D_{\sg}(R)$. For each $n\gg 0$, there exists an $R$-module $M$ such that $X\cong \Sigma^n(M)$ in $\D_{\sg}(R)$ and $\depth(M)\geq \depth(R)$.
\end{lemma}
\begin{proof}
With the same argument in the proof of Lemma \ref{technique}, we may assume $X$ is an $R$-module.  By taking brutal truncation, we see easily that $X$ is isomorphic to $\Sigma^n(\Omega^n_R(X))$ in $\D_{\sg}(R)$ for all $n\in \N$. If  $n\geq \depth(R)$, then $\depth(\Omega^n_R(X))\geq \depth(R)$; see \cite[1.3.7]{BH}. This finishes the proof. 
\end{proof}

For a commutative noetherian local ring $(R,\m, k)$ and a finitely generated $R$-module $M$, we let $\nu(M)$ denote the minimal number of generators of $M$. We let $\ell\ell(R)$ denote the Loewy length of $R$ when $R$ is artinian; see \ref{def of dim}.
\begin{lemma}\label{bound}
Let $(R,\m,k)$ be an isolated singularity and $\dim\D_{\sg}(R)<\infty$. Then

(1) $\ann_R \D_{\sg}(R)$ is $\m$-primary.

(2) For any $\m$-primary ideal $I$ that is contained in $\ann_R\D_{\sg}(R)$, then  
$k$ is a generator of $\D_{\sg}(R)$ with generation time at most $(\nu(I)-\depth(R)+1)\ell\ell(R/I)$. 
\end{lemma}
\begin{proof}
(1) This follows immediately from Theorem \ref{locus}.

(2)  Corollary \ref{iso} yields that $k$ is a generator of $\D_{\sg}(R)$. Since $R/I$ is artinian, $N\in \thick_{\D(R/I)}^{\ell\ell(R/I)}(k)$ for any finitely generated $R/I$-module $N$; see \ref{def of dim}. Restricting scalars along the morphism $R\rightarrow R/I$, we get 
\begin{equation}\label{level}
    N\in \thick_{\D(R)}^{\ell\ell(R/I)}(k)
\end{equation} for any finitely generated $R/I$-module $N$. 

For each $X\in \D_{\sg}(R)$, we claim that  $X\in \thick_{\D_{\sg}(R)}^{(\nu(I)-\depth(R)+1)\ell\ell(R/I)}(k)$. By Lemma \ref{syzygy}, we may assume $X$ is a module and $\depth(X)\geq \depth(R)$
. Choose a minimal set of generators of $I$, say $\bm{x}=x_1,\ldots,x_n$, where $n=\nu(I)$. Since $I\subseteq \ann_R\D_{\sg}(R)$, we get that $X$ is a direct summand of $X\para \bm{x}$ in $\D_{\sg}(R)$. As $I$ is $\m$-primary,  the length of the maximal $X$-regular sequence contained in $I$ is equal to $\depth(X)$. It follows from \cite[Theorem 1.6.17]{BH} that there are at most $n-\depth(X)+1$ cohomologies that are non-zero. Note that each cohomology of $X\para\bm{x}$ is an $R/I$-module. Combining with (\ref{level}), we conclude that $X$ is in $ \thick_{\D_{\sg}(R)}^{(n-\depth(X)+1)\ell\ell(R/I)}(k)$. As $\depth(X)\geq \depth(R)$, we have $$(n-\depth(X)+1)\ell\ell(R/I)\leq (n-\depth(R)+1)\ell\ell(R/I).$$ The desired result follows.
\end{proof}
Combining Remark \ref{finite} with Lemma \ref{bound}, we immediately get the following main result of this section.
\begin{theorem}\label{upper bound}
Let $(R,\m, k)$ be an equicharacteristic excellent local ring. If $R$ has an isolated singularity, then

(1) $\ann_R \D_{\sg}(R)$ is $\m$-primary.

(2) For any $\m$-primary ideal $I$ that is contained in $\ann_R\D_{\sg}(R)$,  then  
$k$ is a generator of $\D_{\sg}(R)$ with generation time at most $(\nu(I)-\depth(R)+1)\ell\ell(R/I)$.  ~~$\square$
\end{theorem}
\begin{remark}\label{connection}
When $(R,\m, k)$ is an equicharacteristic complete Cohen-Macaulay local ring, the above result was proved by Dao and Takahashi \cite[Theorem 1.1]{DT2015} by replacing $\ann_R\D_{\sg}(R)$ by the Noether different of $R$. Indeed, in this case, it is proved that the Noether different  annihilates the singularity category of $R$ and it is $\m$-primary; see \cite[Lemma 2.1, Proposition 4.1]{IT2021} and \cite[Lemma 6.12]{Yoshino}, respectively. Thus we extend Dao and Takahashi's result to the non Cohen-Macaulay rings. 
\end{remark} 

We end this section by applying Theorem \ref{upper bound} to compute an upper bound for the dimension of the singularity category. The ring considered in the following example is not Cohen-Macaulay. Thus one can't apply Dao and Takahashi's result mentioned in Remark \ref{connection}.
\begin{example}
Let $R=k\llbracket x,y\rrbracket/(x^2,xy)$, where $k$ is a field.  This is an equicharacteristic complete local ring. Note that $R$ is not Cohen-Macaulay as $0=\depth(R)<\dim(R)=1$. 

We let $\m$ denote the maximal ideal $(\overline{x},\overline{y})$ of $R$. By Example \ref{example}, we get that $\ann_R\D_{\sg}(R)=\m$. Thus $R$ has an isolated singularity; see Theorem \ref{locus} and Remark \ref{finite}. It follows immediately from Theorem \ref{upper bound} that 
 $$
 \dim\D_{\sg}(R)\leq 3\ell\ell(R/\m)-1=2.
 $$
\end{example}

\bibliographystyle{amsplain}
\bibliography{ref}
\end{document}